\documentclass[10pt, twoside, fleqn, a4paper]{article}

\usepackage{latexsym}  
\usepackage{amscd}    
\usepackage{theorem}  
\usepackage{pifont}  
\usepackage{mathbbol} 
\usepackage{amsfonts}  
\usepackage{xspace}  
\usepackage{amssymb} 
\usepackage{fancyhdr} 
\usepackage{mathrsfs} 
\usepackage{amsmath} 
\usepackage{graphics} 
\usepackage{graphicx} 
\usepackage{euscript}
\usepackage{psfrag}
\usepackage{empheq} 
\usepackage{mathabx} 
\usepackage[parfill]{parskip}     
\usepackage[colorlinks=true, allcolors=blue]{hyperref}
\usepackage{cancel} 

\title{A classification of restrictive polynomial correspondences}
\author{Bharath Krishna Seshadri\footnote{Indian Institute of Science Education and Research Thiruvananthapuram (IISER-TVM). \\ 
ORCID: 0009 - 0004 - 0541 - 8040\ \ email: \texttt{bharathmaths21@iisertvm.ac.in}}\ \ and\ \ Shrihari Sridharan\footnote{Indian Institute of Science Education and Research Thiruvananthapuram (IISER-TVM). \\ ORCID: 0000-0003-2434-4767\ \ email: \texttt{shrihari@iisertvm.ac.in}\ (Corresponding Author) \\ The second author thanks the support provided by NBHM through a research grant \\ No. 02011/35/2025/NBHM(R.P)/R\&D II/9832}}

\DeclareFontFamily{OT1}{pzc}{}
\DeclareFontShape{OT1}{pzc}{m}{it}%
              {<-> s * [0.900] pzcmi7t}{}
\DeclareMathAlphabet{\mathpzc}{OT1}{pzc}%
                                 {m}{it}

{\theorembodyfont{\slshape} \newtheorem{theorem}{Theorem}[section]}
{\theorembodyfont{\slshape} \newtheorem{definition}[theorem]{Definition}}
{\theorembodyfont{\slshape} \newtheorem{lemma}[theorem]{Lemma}}
{\theorembodyfont{\slshape} \newtheorem{proposition}[theorem]{Proposition}}
{\theorembodyfont{\slshape} }
{\theorembodyfont{\slshape} } 
{\theorembodyfont{\slshape} } 
{\theorembodyfont{\slshape} \newtheorem{remark}[theorem]{Remark}}
{\theorembodyfont{\slshape} \newtheorem{example}[theorem]{Example}}

\numberwithin{equation}{section}

\newenvironment{proof}{\paragraph{Proof:}}{\hfill$\bullet$}

\begin{document}

\maketitle

\begin{abstract} 
\noindent 
In this manuscript, we study a special class of correspondences on $\mathbb{P}^{1} \times \mathbb{P}^{1}$ given by a polynomial relation, say $P(z, w)$. We focus on what we call restrictive polynomial correspondence and characterise that it can be written as $P (z, w) = g_{1}(w) h_{1}(z) + \cdots + g_{\rho}(w) h_{\rho}(z)$, for some appropriate $\rho \in \mathbb{Z}_{+}$, where $g_{r}$ and $h_{r}$ are polynomials. In particular, when $\rho = 2$, we say $P$ is irreducible and observe that the equation $P(z, w) = 0$ can be rewritten as $R(z) = S(w)$, where $R$ and $S$ are rational maps of appropriate degree. Further, we also define an operation that, with the exception of degenerate cases, constructs a new irreducible restrictive polynomial correspondence from any two given irreducible restrictive polynomial correspondences.
\end{abstract}

\begin{tabular}{l l} 
{\bf Keywords} & Restrictive polynomial correspondence \\ 
& Full rank factorisation of a matrix \\ 
& Symmetries in a correspondence \\ 
& Reducible and irreducible correspondences \\ 
& \\ 
& \\ 
{\bf MSC Subject} & 12D10, 30C15, 32A08 \\ 
{\bf Classifications} & 
\end{tabular} 
\bigskip 

\section{Introduction}

The study of dynamics of correspondences owes its origins to the papers of Bullett and his collaborators, especially \cite{sb:1988,  bf:2005, bls:2019, bp:1994}. Bullett began his investigations by considering what he called quadratic correspondences. A quadratic correspondence is an implicit relation between two variables in such a way that fixing either of them in the correspondence and finding the zero set of the resulting expression yields a set with cardinality $2$ (counting multiplicity). A simple example of the same is to consider a bivariate polynomial, say $P \in \mathbb{C}[z, w]$ with degree $2$ in both variables. 

Starting with an arbitrary point $z_{(1)} \in \mathbb{C}$, suppose that $w_{(1)}$ and $w_{(2)}$ are the solutions of the equation $P(z_{(1)}, w) = 0$. Then, the point $z_{(1)}$ is necessarily one of the solutions of equations $P(z, w_{(1)}) = 0$ and $P(z, w_{(2)}) = 0$. Suppose that the other of the solutions is $z_{(2)}$ and $z_{(3)}$, respectively. Bullett, in \cite{sb:1988}, defined the polynomial correspondence $P$ to be a map of pairs if for every $z_{(1)} \in \mathbb{C}$, we have $z_{(2)} = z_{(3)}$. He further obtained characterizing conditions for a quadratic polynomial correspondence to be a map of pairs. Though the paper contains far more interesting results concerning the dynamics of such quadratic correspondences, we are currently focused only on the above-mentioned ideas. 

The authors in \cite{sst:2023} generalize the case of quadratic correspondences to include polynomial correspondences of any arbitrary degree $d \ge 2$, although the degree being equal in both the variables to define and characterize, what they call maps of $d$-tuples. A specific hurdle that one encounters while increasing the degree beyond $2$ is the absence of involutions, that Bullett makes extensive use of in his proofs of the characterizing results. The authors, in \cite{sst:2023}, overcome this difficulty by making use of some elementary results from linear algebra. 

Following the work in \cite{sst:2023}, we consider a furthermore generalisation of polynomial correspondences where the degrees in each of the variables need not be equal and complete the analysis of the same. Whenever a polynomial correspondence exhibits a behaviour analogous to the map of pairs under the corresponding situation here, we shall call it a restrictive polynomial correspondence, a precise definition of the same shall appear in Section \ref{sec:rpc}. Moreover, we classify such restrictive polynomial correspondences as reducible and irreducible ones, in Sections \ref{sec:rrpc} and \ref{sec:irpc}, thus exhausting the class of restrictive polynomial correspondences. In Section \ref{sec:*prod}, we define a new multiplication operation between two polynomial correspondences and obtain conditions for two irreducible restrictive polynomial correspondences to result in an irreducible restrictive polynomial correspondence. The paper is strewn with examples illustrating every concept that we introduce, define and characterise. 

We conclude this section by stating that this study of restrictive polynomial correspondences is a small drop in the ocean of holomorphic correspondences, a topic of sufficient interest for many mathematicians in the recent decades, for example \cite{bs:2016, bp3:2001, din:2005, ds:2006}. Towards that end, we start with the definition of a holomorphic correspondence. 

\begin{definition} 
\label{correspondence} 
Consider an implicit relation given by $f(z, w)$, analytic in the variables $z$ and $w$. Let $\mathscr{G}_{f}$ denote the graph of $f$ in $\mathbb{P}^{1} \times \mathbb{P}^{1}$ {\it i.e.}, the biprojective completion of the zero set of $f$ in $\mathbb{P}^{1} \times \mathbb{P}^{1}$. Suppose 
\begin{enumerate} 
\item $\Pi_{w} \left( \Pi_{z}^{-1} \left( \{z_{0}\} \right) \cap \mathscr{G}_{f} \right)$ is generically $d_{z}$-valued for any $z_{0} \in \mathbb{P}^{1}$; 
\item $\Pi_{z} \left( \Pi_{w}^{-1} \left( \{w_{0}\} \right) \cap \mathscr{G}_{f} \right)$ is generically $d_{w}$-valued for any $w_{0} \in \mathbb{P}^{1}$; 
\end{enumerate} 
where $\Pi_{z}$ and $\Pi_{w}$ denote the projection maps defined on $\mathscr{G}_{f}$ along the coordinates respectively, then such an implicit relation $f$ is called \emph{a holomorphic correspondence of bidegree $(d_{w}, d_{z})$}.  
\end{definition} 

Observe that by virtue of Definition \ref{correspondence}, every complex-valued rational map of degree $d$ defined on $\mathbb{P}^{1}$ is now, accorded the status of being a holomorphic correspondence of bidegree $(d, 1)$. We now state a theorem due to Bullett and Penrose, as can be found in \cite{bp3:2001} that characterises holomorphic correspondences and provides us a motivation to consider polynomial correspondences, as we have done in this paper. 

\begin{theorem}[\cite{bp3:2001}, Theorem 2, Page 121] 
Every holomorphic correspondence on the Riemann sphere is algebraic, that is to say, [it] can be expressed in the form $P(z, w) = 0$ for some polynomial $P$ in two variables. 
\end{theorem} 

\section{Restrictive polynomial correspondences} 
\label{sec:rpc} 

Consider a polynomial correspondence $P$ of bidegree $(d_{w}, d_{z})$ in given by 
\begin{equation} 
\label{eqn:pc} 
P_{(d_{w}, d_{z})} \left( z, w \right)\ =\ \sum_{j\, =\, 0}^{d_{w}} \sum_{k\, =\, 0}^{d_{z}} \Lambda_{\left( j, k \right)} z^{j} w^{k}. 
\end{equation} 
Bearing in mind the conditions mentioned in Definition \ref{correspondence}, we shall be interested in those polynomial correspondences $P_{(d_{w}, d_{z})}$ given by Equation \eqref{eqn:pc} that satisfy: 
\begin{enumerate} 
\item[$(P1)$] No linear polynomial in $z$ divides $P_{(d_{w}, d_{z})} (z, w)$; 
\item[$(P2)$] No linear polynomial in $w$ divides $P_{(d_{w}, d_{z})} (z, w)$. 
\end{enumerate} 

In this paper, we consider polynomial correspondences as described above, where $d_{w}$ is not necessarily equal to $d_{z}$, however with the condition that both $d_{w}$ and $d_{z}$ being strictly positive integers with at least one of them being strictly bigger than $1$. To every polynomial correspondence $P_{(d_{w}, d_{z})}$, we associate the matrix $M_{P_{(d_{w}, d_{z})}}$ of order $(d_{w} + 1) \times (d_{z} + 1)$ given by 
\begin{equation} 
\label{eqn:matrix} 
P_{(d_{w}, d_{z})}\ \sim\ M_{P_{(d_{w}, d_{z})}}\ =\ \begin{pmatrix} \Lambda_{\left( d_{w}, d_{z} \right)} & \Lambda_{\left( d_{w}, d_{z} - 1 \right)} & \cdots & \Lambda_{\left( d_{w}, 0 \right)} \\ 
\Lambda_{\left( d_{w} - 1, d_{z} \right)} & \Lambda_{\left( d_{w} - 1, d_{z} - 1 \right)} & \cdots & \Lambda_{\left( d_{w} - 1, 0 \right)} \\ 
\vdots & \vdots & \ddots & \vdots \\ 
\Lambda_{\left( 0, d_{z} \right)} & \Lambda_{\left( 0, d_{z} - 1 \right)} & \cdots & \Lambda_{\left( 0, 0 \right)} 
\end{pmatrix}. 
\end{equation} 

For such polynomial correspondences, we now motivate a situation when it could be restrictive and write the formal definition of a restrictive polynomial correspondence later. For any $z_{0} \in \mathbb{C}$, consider the list of solutions of the polynomial equation $P_{(d_{w}, d_{z})} \left( z_{0}, w \right) = 0$ given by, say $L_{w} (z_{0}) = \left[ w^{(1)}(z_{0}), \cdots, w^{(d_{z})}(z_{0}) \right]$. Substituting $w = w^{(k)}(z_{0})$ for $1 \le k \le d_{z}$ and solving for the polynomial equation $P_{(d_{w}, d_{z})} \left( z, w^{(k)}(z_{0}) \right) = 0$, we obtain a solution list given by 
\[ L_{z} \left(w^{(k)}(z_{0})\right)\ =\ \left[ z^{(1)} \left(w^{(k)}(z_{0})\right), \cdots, z^{(d_{w})} \left(w^{(k)}(z_{0})\right) \right]. \] 
Note that for all $1 \le k \le d_{z}$, at least one of the solutions $z^{(j)} \left(w^{(k)}(z_{0})\right) \equiv z_{0}$. The list of solutions in the $w$-plane naturally depends on the starting point $z_{0}$ and the list of solutions in the $z$-plane depends on our choice of $w^{(k)}(z_{0})$. Suppose the latter list is independent of our choice of $w^{(k)}(z_{0})$, {\it i.e.}, the list of points in $L_{z} \left(w^{(k)}(z_{0})\right)$ remains the same for any $1 \le k \le d_{z}$. This leads us to a situation when a list in the $z$-plane of cardinality $d_{w}$ and a list in the $w$-plane of cardinality $d_{z}$ are associated to each other {\it via} the correspondence. In general, this need not be so. In fact, the list of points associated to each other {\it via} the correspondence $\left[ z^{(j)} \left(w^{(k)}(z_{0})\right) : 1 \le j \le d_{w}; 1 \le k \le d_{z} \right]$ has cardinality $d_{w} d_{z}$. Our ability to reduce this cardinality naturally prompts a thorough study of the same. Towards that direction, is the following definition. 

\begin{definition} 
\label{defn:rescorr} 
Consider a polynomial correspondence $P_{(d_{w}, d_{z})} \left( z, w \right)$ of bidegree $(d_{w}, d_{z})$, as given in Equation \eqref{eqn:pc}. It is said to be a \emph{restrictive polynomial correspondence} if the list of points 
\[ P_{(d_{w}, d_{z})} : L_{z} (w^{(1)}(z_{0}))\ \equiv\ \cdots\ \equiv\ L_{z} (w^{(d_{z})}(z_{0}))\ \ \text{for any}\ z_{0} \in \mathbb{C}. \]
\end{definition} 

As a consequence of Definition \ref{defn:rescorr}, we immediately have the following lemma. 

\begin{lemma} 
\label{rem:abovedefn}
Let $P_{(d_{w}, d_{z})} \left( z, w \right)$ be a restrictive polynomial correspondence of bidegree $(d_{w}, d_{z})$. Then, for any $w_{0} \in \mathbb{C}$ as the starting point, we have 
\[ P_{(d_{w}, d_{z})} : L_{w} (z^{(1)}(w_{0}))\ \equiv\ \cdots\ \equiv\ L_{w} (z^{(d_{w})}(w_{0})). \] 
In other words, our definition of the correspondence as restrictive starting from one of the variables forces a similar restriction on the other variable, too. 
\end{lemma} 

\begin{proof} 
For some fixed $w_{0} \in \mathbb{C}$, let $L_{z} (w_{0}) = \left[ z^{(1)}(w_{0}), \cdots, z^{(d_{w})}(w_{0}) \right]$. Substituting $z = z^{(1)}(w_{0})$ and solving for the polynomial equation $P_{(d_{w}, d_{z})} \left( z^{(1)}(w_{0}), w \right) = 0$, we obtain a solution list given by 
\[ L_{w} \left(z^{(1)}(w_{0})\right)\ =\ \left[ w^{(1)} \left(z^{(1)}(w_{0})\right), \cdots, w^{(d_{z})} \left(z^{(1)}(w_{0})\right) \right], \] 
where $w_{0} \in L_{w} \left(z^{(1)}(w_{0})\right)$ and hence, without loss of generality, we assume $w_{0} = w^{(1)} \left(z^{(1)}(w_{0})\right)$. Since $P_{(d_{w}, d_{z})}$ is restrictive, we have 
\begin{equation}
\label{eqn:rest}
L_{z} (w_{0}) \equiv L_{z} \left(w^{(2)} \left(z^{(1)}(w_{0})\right)\right) \equiv \cdots \equiv L_{z} \left(w^{(d_{z})} \left(z^{(1)}(w_{0})\right)\right).
\end{equation}
Now, fix $1 \leq j \leq d_w$ and consider $P_{(d_{w}, d_{z})} \left( z^{(j)}(w_{0}), w \right) = 0$. By Equation \ref{eqn:rest}, it is clear that, $w^{(k)} \left(z^{(1)}(w_{0})\right)$ is a root of $P_{(d_{w}, d_{z})} \left( z^{(j)}(w_{0}), w \right)$, for all $1 \le k \le d_z$. Further, since there are exactly $d_z$ many roots for $P_{(d_{w}, d_{z})} \left( z^{(j)}(w_{0}), w \right)$, these are the only roots. Thus, we have $L_{w} \left(z^{(1)} (w_{0})\right) \equiv \cdots \equiv L_{w} \left(z^{(d_w)} (w_{0})\right).$
\end{proof} 

\begin{remark} 
\label{rem:lists} 
As a consequence of Definition \eqref{defn:rescorr} and Lemma \eqref{rem:abovedefn}, we can find lists of $d_{w}$ points in the $z$-plane, say $\left[ z_{1}, \cdots, z_{d_{w}} \right]$ and lists of $d_{z}$ points in the $w$-plane, say $\left[ w_{1}, \cdots, w_{d_{z}} \right]$ that are related to each other {\it via} the correspondence $P_{(d_{w}, d_{z})}$ and satisfy $\left(z_{j}, w_{k}\right) \in \mathscr{G}_{P_{(d_{w}, d_{z})}}$. In other words, any starting point $z_{j}$ results in the list of $w_{k}$'s and vice-versa. We denote this by $L_{z} \left( w_{k} \right)\ =\ \left[ z_{1}, \cdots, z_{d_{w}} \right]\ \ \stackrel{P_{(d_{w}, d_{z})}}{\longleftrightarrow}\ \ \left[ w_{1}, \cdots, w_{d_{z}} \right]\ =\ L_{w} \left( z_{j} \right)$. Further, $\left(z, w_{k}\right) \notin \mathscr{G}_{P_{(d_{w}, d_{z})}}$ for $1 \le k \le d_{z}$, unless $z = z_{j}$ for some $1 \le j \le d_{w}$ and $\left(z_{j}, w\right) \notin \mathscr{G}_{P_{(d_{w}, d_{z})}}$ for $1 \le j \le d_{w}$ unless $w = w_{k}$ for some $1 \le k \le d_{z}$. 
\end{remark} 

We now explain the above definition with two examples, that illustrate the concept of restrictive polynomial correspondences. 

\begin{example} 
\label{ex:one} 
Consider the polynomial correspondence of bidegree $(3, 2)$ given by 
\[ P_{(3, 2)} (z, w) = i z^{3} w^{2} + 5 z^{3} w - z^{2} w^{2} + z^{3} - 2 z^{2} w + 11i z w^{2} - 6 z^{2} + 55 z w - w^{2} + 11 z - 2 w - 6. \] 
Observe that $P_{(3, 2)}$ satisfies the required properties, as mentioned in the beginning of this section. Substituting $w \in \left\{ 0, \dfrac{28}{1 - 6i} \right\}$, we obtain $z \in \left\{ 1, 2, 3 \right\}$ and vice-versa. In fact, such an association between $3$ points in the $z$-plane and $2$ points in the $w$-plane {\it via} the polynomial correspondence $P_{(3, 2)}$ can be obtained for any starting point $z_{0}$ or $w_{0}$ in $\mathbb{C}$, thus making $P_{(3, 2)}$ a restrictive polynomial correspondence. 
\end{example} 

\begin{example} 
\label{ex:onesq} 
The polynomial correspondence $P_{(6, 4)} = \left(P_{(3, 2)}\right)^{2}$ of bidegree $(6, 4)$ where $P_{(3, 2)}$ is as stated in Example \ref{ex:one} is also a restrictive polynomial correspondence. However, the association between the points is preserved, in this case, with an increase in the multiplicities of each of the points. 
\end{example} 

\section{Reducible restrictive polynomial correspondences} 
\label{sec:rrpc} 

We begin this section by carefully delineating a distinction between what we describe as the concept of reducibility and irreducibility among restrictive polynomial correspondences, as one might have observed from Examples \ref{ex:one} and \ref{ex:onesq}. We call a restrictive polynomial correspondence $P_{(d_{w}, d_{z})} \left( z, w \right)$, as written in Equation \eqref{eqn:pc}, \emph{irreducible} if it can not be expressed as a product of finitely many (more than one) restrictive polynomial correspondences and \emph{reducible}, otherwise. Thus, the example stated in \ref{ex:one} is an irreducible restrictive polynomial correspondence, while that in \ref{ex:onesq} is a reducible restrictive polynomial correspondence. 

Suppose the polynomial correspondence $P_{(d_{w}, d_{z})}$ can be written as a product of two irreducible restrictive polynomial correspondences, say $P^{(1)}_{(m_{w}, m_{z})}$ and $P^{(2)}_{(n_{w}, n_{z})}$ so that $d_{w} = m_{w} + n_{w}$ and $d_{z} = m_{z} + n_{z}$. We then obtain a necessary condition in order that $P_{(d_{w}, d_{z})}$ is a restrictive polynomial correspondence. For $P_{(d_{w}, d_{z})}$ to be a restrictive polynomial correspondence, we need 
\[ L_{z} \left( w_{k}^{(i)} \right) = \left[ z_{1}^{(1)}, \cdots, z_{m_{w}}^{(1)}, z_{1}^{(2)}, \cdots, z_{n_{w}}^{(2)} \right] \stackrel{P_{(d_{w}, d_{z})}}{\longleftrightarrow} \left[ w_{1}^{(1)}, \cdots, w_{m_{z}}^{(1)}, w_{1}^{(2)}, \cdots, w_{n_{z}}^{(2)} \right] = L_{w} \left( z_{j}^{(i)} \right). \] 

Since $P^{(1)}_{(m_{w}, m_{z})}$ and $P^{(2)}_{(n_{w}, n_{z})}$ are irreducible restrictive polynomial correspondences, by Definition \ref{defn:rescorr}, we know that 
\begin{displaymath} 
\begin{array}{r c l c c c l l l l} 
\vspace{+10pt} 
P^{(1)}_{(m_{w}, m_{z})} & : & L_{z} (w^{(1)}(z_{0})) & \equiv & \cdots & \equiv & L_{z} (w^{(m_{z})}(z_{0})) & \text{for any} & z_{0} \in \mathbb{C}, & \text{and} \\  
P^{(2)}_{(n_{w}, n_{z})} & : & L_{z} (w^{(1)}(z_{0})) & \equiv & \cdots & \equiv & L_{z} (w^{(n_{z})}(z_{0})) & \text{for any} & z_{0} \in \mathbb{C}. & 
\end{array} 
\end{displaymath} 

Owing to Remark \ref{rem:lists}, this forces the lists $\left[ z_{1}^{(1)}, \cdots, z_{m_{w}}^{(1)} \right]$ and $\left[ z_{1}^{(2)}, \cdots, z_{n_{w}}^{(2)} \right]$ to be one and the same and further, the lists $\left[ w_{1}^{(1)}, \cdots, w_{m_{z}}^{(1)} \right]$ and $\left[ w_{1}^{(2)}, \cdots, w_{n_{z}}^{(2)} \right]$ also to be one and the same. This results in the following theorem. 

\begin{theorem} 
\label{thm:factor} 
Any restrictive polynomial correspondence $P_{(d_{w}, d_{z})}$, as stated in Equation \eqref{eqn:pc} can be factorised as $\left[ Q_{(d_{w}', d_{z}')} \right]^{d}$, where $Q_{(d_{w}', d_{z}')}$ is an irreducible restrictive polynomial correspondence and $d$ is a common divisor of $d_{w}$ and $d_{z}$ satisfying $d d_{w}' = d_{w}$ and $d d_{z}' = d_{z}$. 
\end{theorem}

The following example illustrates the fact that the product of two distinct irreducible restrictive polynomial correspondences is not a restrictive polynomial correspondence. 

\begin{example} 
Consider the following irreducible restrictive polynomial correspondences given by 
\[ P^{(1)}_{(3, 2)} (z, w) = z^{3}w^{2} + z^{3} - 6z^{2} + w^{2} + 11z - 6\ \ \ \text{and}\ \ \ P^{(2)}_{(2,2)} (z, w) = z^{2}w^{2} + z^{2} - 2zw + w^{2} - 2z + 1. \] 
Note however, that $P_{(5, 4)} = \left( P^{(1)}_{(3, 2)} \times P^{(2)}_{(2, 2)} \right)$ given by 
\begin{eqnarray*} 
P_{(5, 4)} (z, w) & = & z^{5}w^{4} + 2z^{5}w^{2} - 2z^{4}w^{3} + z^{3}w^{4} - 8z^{4}w^{2} + z^{2}w^{4} + z^{5} - 2z^{4}w + 13z^{3}w^{2} \\ 
& & \hspace{+1cm} - 8z^{4} + 12z^{3}w - 11z^{2}w^{2} - 2zw^{3} + w^{4} + 24z^{3} - 22z^{2}w + 9zw^{2} \\ 
& & \hspace{+1cm} - 34z^{2} + 12zw - 5w^{2} +23z - 6, 
\end{eqnarray*} 
is not a restrictive polynomial correspondence. For example, the list of points $L_{z} = \left[ 1, 2, 3 \right]$ correspond to the list of points $L_{w} = \left[ 0, 0 \right]$ {\it via} $P^{(1)}_{(3, 2)}$, while the list of points $L_{z} = \left[ 1, 1 \right]$ correspond to the list of points $L_{w} = \left[ 0, 1 \right]$ {\it via} $P^{(2)}_{(2, 2)}$. However, the union of the list of points in $L_{z}$ and the union of the list of points in $L_{w}$ for $P^{(1)}_{(3, 2)}$ and $P^{(2)}_{(2,2)}$ do not correspond to the list of points in $L_{z}$ and $L_{w}$ for $P_{(5, 4)}$, {\it i.e.}, $L_{z}\ =\ \left[ 1, 2, 3, 1, 1 \right]\ \ \cancel{\stackrel{P_{(5, 4)}}{\longleftrightarrow}}\ \ \left[ 0, 0, 0, 1 \right]\ =\ L_{w}$. 
\end{example} 

We conclude this section with a result regarding the rank of the matrix pertaining to a reducible restrictive polynomial correspondence. We will prove this theorem, later in Section \ref{sec:thmpf}. 

\begin{theorem} 
\label{thm:conj}
Let $P_{(d_{w}, d_{z})}$, as stated in Equation \eqref{eqn:pc} be a reducible restrictive polynomial correspondence that can be factorised as $P_{(d_{w}, d_{z})} = \left[ Q_{(d_{w}', d_{z}')} \right]^{d}$, where $Q_{(d_{w}', d_{z}')}$ and $d$ are determined as in Theorem \ref{thm:factor}. Then, the rank of the matrix associated to $P_{(d_{w}, d_{z})}$ is $(d + 1)$. 
\end{theorem} 

\section{Irreducible restrictive polynomial correspondences: A characterisation} 
\label{sec:irpc} 

We now turn our focus to irreducible restrictive polynomial correspondences. In this section, we prove an important characterising result about irreducible restrictive polynomial correspondences. 

\begin{theorem} 
\label{thm:tfae}
Consider the polynomial correspondence $P_{(d_{w}, d_{z})}$, as stated in Equation \eqref{eqn:pc}. Then, the following statements are equivalent. 
\begin{enumerate} 
\item $P_{(d_{w}, d_{z})}$ is an irreducible restrictive polynomial correspondence. 
\item The equation $P_{(d_{w}, d_{z})} (z, w) = 0$ can be rewritten by separating the variables $z$ and $w$, as $R (z) = S (w)$, where $R$ and $S$ are rational maps of degrees $d_{w}$ and $d_{z}$ respectively. 
\item The rank of the matrix associated to $P_{(d_{w}, d_{z})}$ is $2$. 
\end{enumerate} 
\end{theorem} 

\begin{proof} 
{\bf 1 $\Longrightarrow$ 2}: Assuming $P_{(d_{w}, d_{z})}$ to be an irreducible restrictive polynomial correspondence, we start by rewriting the equation $P_{(d_{w}, d_{z})} (z, w) = 0$ in two different manners as follows: 
\begin{equation} 
\label{eqn:polypq} 
q_{0}(z) w^{d_{z}} + q_{1}(z) w^{d_{z} - 1} + \cdots + q_{d_{z}}(z)\ =\ 0\ \ \text{and}\ \ p_{0}(w) z^{d_{w}} + p_{1}(w) z^{d_{w} - 1} + \cdots + p_{d_{w}}(w)\ =\ 0, 
\end{equation} 
where $q_{m}$'s are polynomials in $z$ of degree atmost $d_{w}$ and $p_{n}$'s are polynomials in $w$ of degree atmost $d_{z}$. Further, by Remark \ref{rem:lists}, we know that 
\[ L_{z} \left( w_{k} \right)\ =\ \left[ z_{1}, \cdots, z_{d_{w}} \right]\ \ \stackrel{P_{(d_{w}, d_{z})}}{\longleftrightarrow}\ \ \left[ w_{1}, \cdots, w_{d_{z}} \right]\ =\ L_{w} \left( z_{j} \right). \]
Owing to the hypothesis on $P_{(d_{w}, d_{z})}$, there exists at least one $0 \le m \le d_{z}$ such that the degree of the polynomial $q_{m}$ is equal to $d_{w}$. Suppose the degree of the polynomial $q_{0}$ is equal to $d_{w}$, then one may choose any non-zero polynomial $q_{m}$ for $1 \le m \le d_{z}$ and continue with the proof. However, if the degree of $q_{0}$ is not equal to $d_{w}$, then we choose that $1 \le m \le d_{w}$ for which the degree of the polynomial $q_{m}$ is equal to $d_{w}$ and proceed with the proof. Then, Vi\`{e}te's formula (see \cite{vin:2003}, page 89) asserts that 
\[ \frac{q_{m}(z_{j})}{q_{0}(z_{j})}\ =\ \left(-1\right)^{m} \sum_{1\, \le\, i_{1}\, <\, \cdots\, <\, i_{m}\, \le\, d_{z}} w_{i_{1}} \times \cdots \times w_{i_{m}}\ \ \forall z_{j} \in L_{z} \left(w_{k}\right). \]
In fact, the precise solutions to the equation of degree $d_{w}$, namely, 
\[ \dfrac{q_{m}(z)}{q_{0}(z)}\ =\ \left(-1\right)^{m} \sum_{1\, \le\, i_{1}\, <\, \cdots\, <\, i_{m}\, \le\, d_{z}} w_{i_{1}} \times \cdots \times w_{i_{m}} \] 
is given by the $d_{w}$ points $z_{j} \in L_{z} \left(w_{k}\right)$. Similarly, there exists some $1 \le n \le d_{z}$ such that the solutions to the equation of degree $d_{z}$ given by 
\[ \frac{p_{n}(w)}{p_{0}(w)}\ =\ \left(-1\right)^{n} \sum_{1\, \le\, i_{1}\, <\, \cdots\, <\, i_{n}\, \le\, d_{w}} z_{i_{1}} \times \cdots \times z_{i_{n}} \] 
is precisely the $d_{z}$ points $w_{k} \in L_{w} \left(z_{j}\right)$. Note that $\dfrac{q_{m}}{q_{0}}$ and $\dfrac{p_{n}}{p_{0}}$ are rational maps in the variables $z$ and $w$ of degree $d_{w}$ and $d_{z}$ respectively. Finally, we make use of an automorphism of the Riemann sphere, namely a M\"{o}bius map $M$, and define $R = M \circ \dfrac{q_{m}}{q_{0}}$ and $S = \dfrac{p_{n}}{p_{0}}$ to complete the proof, in this part. 

{\bf 2 $\Longrightarrow$ 3}: Owing to the hypothesis, we assume that the equation $P_{(d_{w}, d_{z})} (z, w) = 0$ that yields the zero set of the correspondence can be alternately expressed as 
\begin{equation} 
\label{eqn:sep} 
R(z)\ =\ \frac{A_{d_{w}} z^{d_{w}} + \cdots + A_{0}}{B_{d_{w}} z^{d_{w}} + \cdots + B_{0}}\ =\ \frac{C_{d_{z}} w^{d_{z}} + \cdots + C_{0}}{D_{d_{z}} w^{d_{z}} + \cdots + D_{0}}\ =\ S(w), 
\end{equation} 
where $A_{j}, B_{j}, C_{k}, D_{k} \in \mathbb{C}$. Then, the matrix $M_{P_{(d_{w}, d_{z})}}$ of order $(d_{w} + 1) \times (d_{z} + 1)$ pertaining to the correspondence is given by 
\[ \begin{pmatrix} A_{d_{w}} D_{d_{z}} - B_{d_{w}} C_{d_{z}} & A_{d_{w}} D_{d_{z} - 1} - B_{d_{w}} C_{d_{z} - 1} & \cdots & A_{d_{w}} D_{0} - B_{d_{w}} C_{0} \\ 
A_{d_{w} - 1} D_{d_{z}} - B_{d_{w} - 1} C_{d_{z}} & A_{d_{w} - 1} D_{d_{z} - 1} - B_{d_{w} - 1} C_{d_{z} - 1} & \cdots & A_{d_{w} - 1} D_{0} - B_{d_{w} - 1} C_{0} \\ 
\vdots & \vdots & \ddots & \vdots \\ 
A_{0} D_{d_{z}} - B_{0} C_{d_{z}} & A_{0} D_{d_{z} - 1} - B_{0} C_{d_{z} - 1} & \cdots & A_{0} D_{0} - B_{0} C_{0} 
\end{pmatrix}, \] 
that can be factorised as the product of two matrices, say $\mathcal{R}$ and $\mathcal{S}$ of sizes $(d_{w} + 1) \times 2$ and $2 \times (d_{z} + 1)$ respectively as 
\[ \mathcal{R} \mathcal{S}\ =\ \begin{pmatrix} A_{d_{w}} & B_{d_{w}} \\ A_{d_{w} - 1} & B_{d_{w} - 1} \\ \vdots & \vdots \\ A_{0} & B_{0} \end{pmatrix} \begin{pmatrix} D_{d_{z}} & D_{d_{z} - 1} & \cdots & D_{0} \\ -C_{d_{z}} & -C_{d_{z} - 1} & \cdots & -C_{0} \end{pmatrix}. \]  
Since $R$ and $S$ are non-constant rational maps of degree $d_{w}$ and $d_{z}$ respectively, the matrices $\mathcal{R}$ and $\mathcal{S}$ are necessarily of rank $2$. This forces that the rank of $M_{P_{(d_{w}, d_{z})}}$ is atmost $2$. It is a simple observation that ${\rm Rank} \left(M_{P_{(d_{w}, d_{z})}}\right) \ne 0$, since it is not the zero matrix. Thus, we prove that the ${\rm Rank} \left(M_{P_{(d_{w}, d_{z})}}\right) \ne 1$ to complete the proof of this part. 

We start by assuming that ${\rm Rank} \left(M_{P_{(d_{w}, d_{z})}}\right) = 1$ and end up with a contradiction. Then, one may suppose that there exists one column, say ${\rm Col}_{j}$ for $0 \le j \le d_{w}$ such that every other column in $M_{P_{(d_{w}, d_{z})}}$ is a scalar multiple of ${\rm Col}_{j}$. This then gives rise to a factorisation of the equation $P_{(d_{w}, d_{z})} = 0$ as $q_{j}(z) T(w) = 0$, where $q_{j}$ is the polynomial, as mentioned in Equation \eqref{eqn:polypq} and $T$ is a polynomial in $w$. This is a contradiction to the properties $(P1)$ and $(P2)$ of the correspondence, as stated in Section \ref{sec:rpc}. 

{\bf 3 $\Longrightarrow$ 2} By hypothesis, we consider the polynomial correspondence $P_{(d_{w}, d_{z})}$, as stated in Equation \eqref{eqn:pc} such that the rank of the matrix associated to $P_{(d_{w}, d_{z})}$ is $2$. The strategy of the proof is to employ the full rank factorisation theorem, as can be found in, say \cite{hj:2013, po:1999}, that asserts that there exists two matrices, say $\mathcal{R}$ and $\mathcal{S}$ of orders $(d_{w} + 1) \times 2$ and $2 \times (d_{z} + 1)$ each of whose rank is $2$, so that $M_{P_{(d_{w}, d_{z})}} = \mathcal{R} \mathcal{S}$ and to construct the rational maps $R$ and $S$ using the matrices $\mathcal{R}$ and $\mathcal{S}$. 

Note that there exists two columns, say ${\rm Col}_{j}$ and ${\rm Col}_{k}$, in $M_{P_{(d_{w}, d_{z})}}$ that form a basis for the column space in order that every column of the matrix $M_{P_{(d_{w}, d_{z})}}$ can be written as a linear combination of ${\rm Col}_{j}$ and ${\rm Col}_{k}$, {\it i.e.}, ${\rm Col}_{m} = D_{m} {\rm Col}_{j} - C_{m} {\rm Col}_{k}\ \forall 0 \le m \le d_{z}$. Define 
\[ \mathcal{R}\ =\ \begin{pmatrix} {\rm Col}_{j} & {\rm Col}_{k} \end{pmatrix}\ \ \ \ \mathcal{S}\ =\ \begin{pmatrix} D_{d_{z}} & \cdots & D_{0} \\ -C_{d_{z}} & \cdots & -C_{0} \end{pmatrix}, \]
and consider the rational maps pertaining to the matrices $\mathcal{R}$ and $\mathcal{S}$ given by 
\[ R(z)\ =\ \frac{A_{d_{w}} z^{d_{w}} + \cdots + A_{0}}{B_{d_{w}} z^{d_{w}} + \cdots + B_{0}}\ \ \text{and}\ \ S(w)\ =\ \frac{C_{d_{z}} w^{d_{z}} + \cdots + C_{0}}{D_{d_{z}} w^{d_{z}} + \cdots + D_{0}}, \] 
where ${\rm Col}_{j} = \begin{pmatrix} A_{n} \end{pmatrix}_{d_{w}\, \ge\, n\, \ge\, 0}$ and ${\rm Col}_{k} = \begin{pmatrix} B_{n} \end{pmatrix}_{d_{w}\, \ge\, n\, \ge\, 0}$, thereby enabling us to write the equation $P_{(d_{w}, d_{z})} (z, w)= 0$ as $R(z) = S(w)$. 

{\bf 2 $\Longrightarrow$ 1} Owing to the hypothesis, we first observe that $P_{(d_{w}, d_{z})}$ is a restrictive polynomial correspondence, by definition, meaning $L_{z} \left( w_{k} \right) = \left[ z_{1}, \cdots, z_{d_{w}} \right] \stackrel{P_{(d_{w}, d_{z})}}{\longleftrightarrow} \left[ w_{1}, \cdots, w_{d_{z}} \right] = L_{w} \left( z_{j} \right)$. In order to prove that $P_{(d_{w}, d_{z})}$ represents an irreducible restrictive polynomial correspondence by assuming that the equation $P_{(d_{w}, d_{z})} = 0$ can be rewritten as $R(z) = S(w)$, we prove the contra-positive statement. 

On that note, we assume $P_{(d_{w}, d_{z})}$ to be a reducible restrictive polynomial correspondence. Then, by Theorem \ref{thm:factor} one can factorise $P_{(d_{w}, d_{z})}$ as $\left[ Q_{(d_{w}', d_{z}')} \right]^{d}$, where $Q_{(d_{w}', d_{z}')}$ is an irreducible restrictive polynomial correspondence and $d > 1$ is a common divisor of $d_{w}$ and $d_{z}$ satisfying $d d_{w}' = d_{w}$ and $d d_{z}' = d_{z}$. Further, from the $1 \Longrightarrow 2$ part of the current theorem, one can write the equation $Q_{(d_{w}', d_{z}')} = 0$ as $R_{1}(z) = S_{1}(w)$, where $R_{1}$ and $S_{1}$ are rational maps of appropriate degrees. Observing that the equation $\left[ R_{1}(z) - S_{1}(w) \right]^{d} = 0$ can never be written as $R(z) = S(w)$ for any rational maps $R$ and $S$, whenever $d > 1$, finally completes the proof. 
\end{proof} 

As written in the proof of the last implication, even though $\left[ Q_{(d_{w}', d_{z}')} \right]^{d} = 0$ can never be written as $R(z) = S(w)$ for any rational maps $R$ and $S$, when $d > 1$, Lemma \ref{lem:rho} describes the precise alternate expression for such correspondences. 

One obtains the results that characterise, what the authors call a map of $d$-tuples in \cite{sst:2023}, as a corollary to Theorem \ref{thm:tfae}, when $d_{z} = d_{w} = d \ge 2$. Moreover, the specific case when $d = 2$ was investigated by Bullett in \cite{sb:1988}. 

\begin{example} 
Recall the correspondence $P_{(3, 2)}$ in Example \eqref{ex:one}, as stated in Section \ref{sec:rpc}. Then, note that the equation $P_{(3, 2)} (z, w) = 0$ can be rewritten as $\dfrac{z^{3} + 11 z}{z^{2} + 1} = \dfrac{w^{2} + 2 w + 6}{i w^{2} + 5 w + 1}$. Further, the matrix $M_{P_{(3, 2)}}$ pertaining to the given correspondence has rank $2$.  
\end{example} 

The next two remarks concern a certain symmetry that one may be interested in polynomial correspondences. 

\begin{remark} 
\begin{enumerate} 
\item Suppose $P_{(d_{w}, d_{z})}$ is an irreducible restrictive polynomial correspondence that satisfies $P_{(d_{w}, d_{z})} (z, w) = 0$ {\it iff} $P_{(d_{w}, d_{z})} (-z, -w) = 0$. Then, any two consecutive columns of the matrix $M_{P_{(d_{w}, d_{z})}}$ pertaining to the correspondence are linearly independent, provided neither of them is the zero vector. 
\item Suppose in Equation \eqref{eqn:sep}, we have the quantities $A_{j}, B_{j}, C_{k}$ and $D_{k}$ for all $0 \le j \le d_{w}$ and for all $0 \le k \le d_{z}$ to be real. Then, $P_{(d_{w}, d_{z})} (z, w) = 0$ {\it iff} $P_{(d_{w}, d_{z})} (\overline{z}, \overline{w}) = 0$. The converse is also true, in the sense that suppose the irreducible restrictive polynomial correspondence satisfies $P_{(d_{w}, d_{z})} (z, w) = 0$ {\it iff} $P_{(d_{w}, d_{z})} (\overline{z}, \overline{w}) = 0$. Then, $P_{(d_{w}, d_{z})} \in \mathbb{R}[z, w]$, upto a constant multiple. 
\end{enumerate} 
\end{remark} 

\section{$*$-product of irreducible restrictive polynomial correspondences} 
\label{sec:*prod} 

In this section, we consider two irreducible restrictive polynomial correspondences, define an operation between them and investigate when the resulting expression yields something interesting, possibly an irreducible restrictive polynomial correspondence of appropriate bidegree. Towards that end, consider $P_{(m_{w}, m_{z})}$ and $Q_{(n_{w}, n_{z})}$ with $m_{z} = n_{w} = d$ given by 
\[ P_{(m_{w}, d)} \left( z, w \right)\ =\ \sum_{j\, =\, 0}^{m_{w}} \sum_{k\, =\, 0}^{d} \Lambda_{\left( j, k \right)} z^{j} w^{k}\ \ \text{and}\ \ Q_{(d, n_{z})} \left( z, w \right)\ =\ \sum_{j\, =\, 0}^{d} \sum_{k\, =\, 0}^{n_{z}} \Omega_{\left( j, k \right)} z^{j} w^{k}. \] 

As mentioned in Equation \eqref{eqn:polypq}, one can then express these correspondences $P_{(m_{w}, d)}$ and $Q_{(d, n_{z})}$ alternatively as 
\begin{equation} 
\label{eqn:polypqo} 
P_{(m_{w}, d)} \left( z, w \right)\ =\ \sum_{j\, =\, 0}^{m_{w}} p_{j} (w) z^{j}\ \ \text{and}\ \ Q_{(d, n_{z})} \left( z, w \right)\ =\ \sum_{k\, =\, 0}^{n_{z}} q_{k} (z) w^{k}, 
\end{equation} 
where $p_{j}$'s and $q_{k}$'s are polynomials of degree atmost $d$ in $w$ and $z$ respectively given by $\displaystyle{p_{j} (w) = \sum_{r\, =\, 0}^{d} \alpha_{r}^{(j)} w^{r}}$ for $0 \le j \le m_{w}$ and $\displaystyle{q_{k} (z) = \sum_{r\, =\, 0}^{d} \beta_{r}^{(k)} z^{r}}$ for $0 \le k \le n_{z}$. For such a collection of polynomials consisting of $p_{j}$'s and $q_{k}$'s, define the $*$-product of $P_{(m_{w}, d)}$ and $Q_{(d, n_{z})}$ as: 
\begin{equation} 
\label{defn:star} 
\left( P_{(m_{w}, d)} * Q_{(d, n_{z})} \right) (z, w)\ =\ \sum_{j\, =\, 0}^{m_{w}} \sum_{k\, =\, 0}^{n_{z}} \langle p_{j}, q_{k} \rangle z^{j} w^{k}\ \ \text{where}\ \ \langle p_{j}, q_{k} \rangle\ \ =\ \ \sum_{r\, =\, 0}^{d} \alpha_{r}^{(j)} \beta_{r}^{(k)}. 
\end{equation} 

We first discuss the cases when the $*$-product between any two irreducible restrictive polynomial correspondences results in an expression that violates the conditions $(P1)$ and $(P2)$, as mentioned in Section \ref{sec:rpc} for polynomial correspondences and hence, are not desirable. Consider the following examples. 

\begin{example} 
\label{ex:zero} 
Given two irreducible restrictive polynomial correspondences, 
\begin{eqnarray*} 
P_{(3, 4)} (z, w) & = & 2z^{3}w^{4} + z^{3}w^{3} - w^{4} + 7z^{3} - 7w^{3} - w^{2} - w + 4\ \ \text{and} \\ 
Q_{(4, 2)} (z, w) & = & z^{4}w^{2} - 2z^{3}w^{2} - 53z^{4} + 13z^{2}w^{2} + 15z^{3} + 13. 
\end{eqnarray*} 
We note that $\langle p_{j}, q_{k} \rangle = 0$ for all $j, k$, that yields $\left( P_{(3, 4)} * Q_{(4, 2)} \right) \equiv 0$. 
\end{example} 

\begin{example} 
\label{ex:zintow} 
Let 
\begin{eqnarray*} 
P_{(3, 2)} (z, w) & = & 2iz^{3}w^{2} + 3z^{3}w + z^{3} + 5z^{2}w + 2izw^{2} + 3zw + z + 2iw^{2} + 28w + 1\ \ \text{and} \\ 
Q_{(2, 3)} (z, w) & = & z^{2}w^{3} + 1  
\end{eqnarray*} 
be the two given irreducible restrictive polynomial correspondences. Then, we obtain 
\begin{eqnarray*} 
\left( P_{(3, 2)} * Q_{(2, 3)} \right) (z, w) & = & 2iz^{3}w^{3} + z^{3} + 2izw^{3} + z + 2iw^{3} + 1 \\ 
& = & \left( z^{3} + z + 1 \right) \left( 2iw^{3} + 1 \right). 
\end{eqnarray*} 
In this example, $P_{(3, 2)} * Q_{(2, 3)}$ is undesirable since it disobeys the conditions $(P1)$ and $(P2)$, as mentioned in Section \ref{sec:rpc} for polynomial correspondences. 
\end{example} 

Since the $*$-product of correspondences have reduced to being uninteresting in the above two examples, we examine the $*$-product further using alternate expressions of $P_{(m_{w}, d)}$ and $Q_{(d, n_{z})}$. Rewriting the equation $P_{(m_{w}, d)} \left(z, w\right) = 0$, by virtue of Theorem \ref{thm:tfae}, we have 
\begin{equation} 
\label{eqn:rationalforp} 
R_{1}(z)\ =\ \frac{A_{m_{w}} z^{m_{w}} + \cdots + A_{0}}{B_{m_{w}} z^{m_{w}} + \cdots + B_{0}}\ \ =\ \ \frac{C_{d} w^{d} + \cdots + C_{0}}{D_{d} w^{d} + \cdots + D_{0}}\ =\ S_{1}(w). 
\end{equation} 
Similarly, the equation $Q_{(d, n_{z})} \left(z, w\right) = 0$ can be rewritten as 
\begin{equation} 
\label{eqn:rationalforq} 
R_{2}(z)\ =\ \frac{E_{d} z^{d} + \cdots + E_{0}}{F_{d} z^{d} + \cdots + F_{0}}\ \ =\ \ \frac{G_{n_{z}} w^{n_{z}} + \cdots + G_{0}}{H_{n_{z}} w^{n_{z}} + \cdots + H_{0}}\ =\ S_{2}(w). 
\end{equation} 
We now define four bivariate polynomials denoted by $f_{i}$ for $1 \le i \le 4$, given by 
\begin{equation} 
\label{eqn:fi} 
\begin{array}{r c l r c l} 
f_{1} (z, w) & = & {\rm Nr.} \left(S_{1}\right) {\rm Dr.} \left(R_{2}\right), & f_{2} (z, w) & = & {\rm Dr.} \left(S_{1}\right) {\rm Nr.} \left(R_{2}\right), \vspace{+5pt} \\ 
f_{3} (z, w) & = & {\rm Nr.} \left(S_{1}\right) {\rm Nr.} \left(R_{2}\right), & f_{4} (z, w) & = & {\rm Dr.} \left(S_{1}\right) {\rm Dr.} \left(R_{2}\right), 
\end{array} 
\end{equation} 
where ${\rm Nr.}$ and ${\rm Dr.}$ denotes the numerator and the denominator of the appropriate rational function, as written above. Note that each $f_{i} (z, w)$ represents a bivariate polynomial of bidegree $(d, d)$ and hence, is associated to a matrix of size $(d + 1) \times (d + 1)$, as written in Equation \eqref{eqn:matrix}. Denoting by ${\rm Tr} (f_{i})$, the trace of the matrix associated to $f_{i}$, we now define 
\[ S(w)\ =\ \frac{\left[H_{n_{z}} {\rm Tr} (f_{3}) - G_{n_{z}} {\rm Tr} (f_{1})\right] w^{n_{z}} + \cdots + \left[H_{0} {\rm Tr} (f_{3}) - G_{0} {\rm Tr} (f_{1})\right]}{\left[H_{n_{z}} {\rm Tr} (f_{2}) - G_{n_{z}} {\rm Tr} (f_{4})\right] w^{n_{z}} + \cdots + \left[H_{0} {\rm Tr} (f_{2}) - G_{0} {\rm Tr} (f_{4})\right]}. \]
We let $R = R_{1}$ and consider the equation $R(z) = S(w)$ that can be re-written as 
\begin{equation} 
\label{eqn:tea} 
T (z, w) = 0. 
\end{equation} 

We first prove that $\left( P_{(m_{w}, d)} * Q_{(d, n_{z})} \right) = T$. In order to do the same, we prove that the coefficient of the $z^{j} w^{k}$ in both the sides of the above equation are equal for any $0 \le j \le m_{w}$ and $0 \le k \le n_{z}$. Owing to the representation of $P_{(m_{w}, d)}$ and $Q_{(d, n_{z})}$, as given in Equation \eqref{eqn:polypqo} and the alternative manner in which the equations $P_{(m_{w}, d)} (z, w) = 0$ and $Q_{(d, n_{z})} (z, w) = 0$ can be written, as mentioned in Equations \eqref{eqn:rationalforp} and \eqref{eqn:rationalforq}, we note that for $0 \le j \le m_{w}$ and $0 \le k \le n_{z}$, we have 
\begin{eqnarray*} 
\langle p_{j}, q_{k} \rangle & = & A_{j} \left( H_{k} {\rm Tr} (f_{2}) - G_{k} {\rm Tr} (f_{4}) \right) - B_{j} \left( H_{k} {\rm Tr} (f_{3}) - G_{k} {\rm Tr} (f_{1}) \right) \\ 
& = & \text{Coefficient of $z^{j} w^{k}$ in $T (z, w)$}. 
\end{eqnarray*} 

Making use of this new representation for the $*$-product, we reexamine Examples \ref{ex:zero} and \ref{ex:zintow} to understand why they gave rise to undesirable $*$-products. One can note that in Example \ref{ex:zero}, we obtain $S(w)$ as an indeterminate, while in Example \ref{ex:zintow}, ${\rm Nr.} (S) = 0$. Further, we note that we can avoid these undesirabilities by demanding the following conditions: 
\[ \langle {\rm Nr.} (S_{1}), q_{0} \rangle\ \ \ne\ \ 0\ \ \ \ \text{and}\ \ \ \ \langle {\rm Dr.} (S_{1}), q_{n_{z}} \rangle\ \ \ne\ \ 0. \] 

In addition to the cases explained above, it is quite possible that $S(w)$ turns out to be a scalar quantity. This results in $\dfrac{H_{k}}{G_{k}} = \alpha$ for all $0 \le k \le n_{z}$, for some $\alpha \in \mathbb{C}$. However, this is not possible, since $S_{2}$ is a rational map of degree $n_{z}$. The only other possibility by way of which we may end up with $S(w)$ being a scalar quantity is when there exists a $\beta \in \mathbb{C}$ such that ${\rm Tr}(f_{2}) - \beta {\rm Tr}(f_{3}) = 0$ and ${\rm Tr}(f_{4}) - \beta {\rm Tr}(f_{1}) = 0$. However, if $S(w)$ turns out to be a scalar, we will lose one variable in our expression of $T(z, w)$. We provide one such example, here. 

\begin{example} 
Consider the irreducible restrictive polynomial correspondences 
\begin{eqnarray*} 
P_{(3,4)} (z,w) & = & 2 z^{3} w^{4} + z^{3} w^{3} - w^{4} + 7 z^{3} - 7 w^{3} - w^{2} - w + 4\ \ \text{and} \\ 
Q_{(4,2)} (z,w) & = & z^{4} w^{2} - 2 z^{3} w^{2} - 53 z^{4} + 13 z^{2} w^{2} - 15 z^{3}  + 13, 
\end{eqnarray*} 
which when equated to $0$ can be alternately expressed as 
\[ z^{3}\ = \dfrac{w^{4} + 7 w^{3} + w^{2} + w - 4}{2 w^{4} + w^{3} + 7}\ \ \text{and}\ \ \dfrac{53 z^{4} + 15 z^{3} - 13}{z^{4} - 2 z^{3} + 13 z^{2}}\ =\ w^{2}, \] 
respectively. We then observe in this case that 
\[ {\rm Tr}(f_{1})\ =\ 0,\ \ {\rm Tr}(f_{2})\ =\ 30,\ \ {\rm Tr}(f_{3})\ =\ 210,\ \ \text{and}\ \ {\rm Tr}(f_{4})\ =\ 0. \] 
Choosing $\beta = \dfrac{1}{7}$, we see that ${\rm Tr}(f_{2}) = \beta {\rm Tr}(f_{3})$ and ${\rm Tr}(f_{4}) = \beta {\rm Tr}(f_{1})$, thus, reducing $S(w)$ to be a constant function, a case that we prefer to avoid, since the degree of the $w$-variable in the polynomial is zero and hence, $T(z, w)$ is not even a correspondence, according to the conditions mentioned in Section \ref{sec:rpc}. 
\end{example} 

We now focus on the conditions to obtain $T(z, w)$ as an irreducible restrictive polynomial correspondence, by making such demands as appropriate to overcome the above mentioned limitations. In what follows, we consider irreducible restrictive polynomial correspondences $P_{(m_{w}, d)}$ and $Q_{(d, n_{z})}$. The polynomials $p_{j}$'s and $q_{k}$'s are those given in Equation \ref{eqn:polypqo}, the rational maps $R_{1}, R_{2}, S_{1}$ and $S_{2}$ are the ones mentioned in Equation \ref{eqn:rationalforp}. We also consider the four bivariate polynomials $f_{i} (z, w)$, where $i = 1, 2, 3, 4$ as in Equation \ref{eqn:fi}. 

\begin{definition}
A pair of irreducible restrictive polynomial correspondences $P_{(m_{w}, d)}$ and $Q_{(d, n_{z})}$ is said to be \emph{non-degenerate} if it satisfies the following criteria: 
\begin{enumerate}
\item $\langle {\rm Nr.} (S_{1}), q_{0} \rangle \ne 0\ \ \ \text{and}\ \ \ \langle {\rm Dr.} (S_{1}), q_{n_{z}} \rangle \ne 0$; 
\item There exists no $\beta \in \mathbb{C}$ that satisfies ${\rm Tr}(f_{2}) = \beta {\rm Tr}(f_{3})$ and ${\rm Tr}(f_{4}) = \beta {\rm Tr}(f_{1})$ simultaneously; 
\item $\langle p_{m_{w}}, q_{n_{z}} \rangle \ne 0$ and $\langle p_{0}, q_{0} \rangle \ne 0$. 
\end{enumerate}
\end{definition}

Then, we have the following theorem. 

\begin{theorem} 
\label{thm:*mult} 
Let $P_{(m_{w}, d)}$ and $Q_{(d, n_{z})}$ be a pair of non-degenerate irreducible restrictive polynomial correspondences. Then, $T = P_{(m_{w}, d)} * Q_{(d, n_{z})}$ is an irreducible restrictive polynomial correspondence of bidegree $(m_{w}, n_{z})$. 
\end{theorem}

Since the genesis of the conditions mentioned in the hypothesis of the theorem have been explained already, an appeal to Theorem \ref{thm:tfae} completes the proof of Theorem \ref{thm:*mult}. We also note that instead of fixing $R(z) = R_{1}(z)$ and defining $S(w)$ using the bivariate polynomials $f_{i}$, for $1 \le i \le 4$, one may as well fix $S(w) = S_{2}(w)$ and obtain $R(z)$ using a new set of bivariate polynomials, say $g_{i}$ for $1 \le i \le 4$, by means of the other three rational maps $R_{1}, S_{1}$ and $R_{2}$. We now write an example to illustrate the idea that Theorem \ref{thm:*mult} propounds. 

\begin{example} 
Consider the following irreducible restrictive polynomial correspondences 
\begin{eqnarray*}
P_{(3, 4)} (z,w) & = & 2 z^{3} w^{4} + i z^{3} w^{3} + 3 z^{2} w^{4} + (7 + i) z^{2} w^{3} + z^{2} w^{2} - 2i w^{4} + 7 z^{3} + 2 z^{2} w \\ 
& & \hspace{+1.2cm} + w^{3} + 3 z^{2} - 7i, \\ 
Q_{(4, 5)} (z,w) & = & 2 z^{4} w^{4} + z^{3} w^{5} + 5 z^{3} w^{4} +3 z^{2} w^{5} + 4 z^{4} w^{2} + i z^{3} w^{3} + 13 z^{2} w^{4} + z w^{5} \\ 
& & \hspace{+1.2cm} + (4 + 6i) z^{3} w^{2} + 3i z^{2} w^{3} + 3 z w^{4} - w^{5} + z^{4} + z^{3} w \\ 
& & \hspace{+1.2cm} + (8 + 18i) z^{2} w^{2} + i z w^{3} + (2i - 3) w^{4} + 2 z^{3} + 3 z^{2} w + 6i z w^{2} \\ 
& & \hspace{+1.2cm} - i w^{3} + 5 z^{2} + z w - 2i w^{2} + z - w + (i - 1). 
\end{eqnarray*}

After due computations, one finds that 
\begin{eqnarray*}
\left( P_{(3,4)} * Q_{(4,5)} \right) (z, w) & = & (-7+i) z^{3} w^{5} + (-17+19i) z^{3} w^{4} + (9+i) z^{2} w^{5} - (1+7i) z^{3} w^{3} \\ 
        &&+ (51+11i) z^{2} w^{4} + (2-10i) z^{3} w^{2} + (-1+9i) z^{2} w^{3} + (1+7i) w^{5} \\
        &&+ (-7+i) z^{3} w + (42+70i) z^{2} w^{2} + (19+17i) w^{4} + (-5+9i) z^{3} \\
        &&+ (9+i) z^{2} w + (-7+i) w^{3} + (21+5i) z^{2} - (10+2i) w^{2} \\
        &&+ (1+7i) w + (9+5i) 
\end{eqnarray*}
is an irreducible restrictive polynomial correspondence, courtesy Theorem \ref{thm:tfae}, since the matrix corresponding to $(P_{(3,4)} * Q_{(4,5)})$ has rank $2$. 
\end{example} 

We now state and prove a proposition that provides a partial converse to Theorem \ref{thm:*mult}. 

\begin{proposition} 
\label{remtoprop}
Any irreducible restrictive polynomial correspondence $P_{(d_{w}, d_{z})}$ can be written as $P_{(d_{w}, d_{z})} (z, w)\ =\  \left( P^{(1)}_{(d_{w}, 1)} * P^{(2)}_{(1, d_{z})} \right) (z, w)$, where $P^{(1)}_{(d_{w}, 1)}$ and $P^{(2)}_{(1, d_{z})}$ are maps on the appropriate variables, masquerading as irreducible restrictive polynomial correspondences. 
\end{proposition} 

\begin{proof} 
Let $P_{(d_{w}, d_{z})} (z, w) = 0$ be rewritten as in Equation \eqref{eqn:sep}. Inspired by the ideas used in Theorem \ref{thm:tfae}, we consider the irreducible restrictive polynomial correspondences $P^{(1)}_{(d_{w}, 1)}$ and $P^{(2)}_{(1, d_{z})}$, respectively, obtained from equations
\[ \frac{A_{d_{w}} z^{d_{w}} + \cdots + A_{0}}{B_{d_{w}} z^{d_{w}} + \cdots + B_{0}}\ \ =\ \ \frac{C_{d_{z}} w + C_{0}}{D_{d_{z}} w + D_{0}}\ \ \ \ \text{and}\ \ \ \ z\ \ =\ \ \frac{- \frac{J_{(d_{z} - 1, d_{z})}}{J_{(0, d_{z})}} w^{d_{z} - 1} - \cdots - \frac{J_{(d_{z} - 1, 1)}}{J_{(0, d_{z})}} w - 1}{w^{d_{z}} + \frac{J_{(0, d_{z} - 1)}}{J_{(0, d_{z})}} w^{d_{z} - 1} + \cdots + \frac{J_{(0, 1)}}{J_{(0, d_{z})}} w}, \]
where $J_{m, n} = D_{m} C_{n} - D_{n} C_{m}$. One can then verify that $\left( P^{(1)}_{(d_{w}, 1)} * P^{(2)}_{(1, d_{z})} \right) = P_{(d_{w}, d_{z})}$. 
\end{proof} 

We conclude this section with a remark on two different symmetries that one may explore when dealing with irreducible restrictive polynomial correspondences; with the focus being when $d_{w} \ne d_{z}$. 

\begin{remark} 
Let $P_{(d_{w}, d_{z})}$ be any irreducible restrictive polynomial correspondence. Suppose that the transpose (and the conjugate transpose) of $M_{P_{(d_{w}, d_{z})}}$, given by $M^{{\rm T}}_{P_{(d_{w}, d_{z})}}$ (and $\overline{M}^{{\rm T}}_{P_{(d_{w}, d_{z})}}$) of order $(d_{z} + 1) \times (d_{w} + 1)$ pertains to the polynomial correspondence, denoted by $^{\dagger}\hspace{-2pt}P_{(d_{z}, d_{w})} (z, w)$ (and $^{\dagger}\hspace{-2pt}\overline{P}_{(d_{z}, d_{w})} (z, w)$, respectively) of bidegree $(d_{z}, d_{w})$. Then, 
\begin{enumerate} 
\item The matrix pertaining to the correspondence $P_{(d_{w}, d_{z})} * ^{\dagger}\hspace{-2pt}P_{(d_{z}, d_{w})}$, {\it i.e.}, $M_{(P_{(d_{w}, d_{z})} * ^{\dagger}\hspace{-2pt}P_{(d_{z}, d_{w})})}$ is symmetric. 
\item The matrix pertaining to the correspondence $P_{(d_{w}, d_{z})} * ^{\dagger}\hspace{-2pt}\overline{P}_{(d_{z}, d_{w})}$, {\it i.e.}, $M_{(P_{(d_{w}, d_{z})} * ^{\dagger}\hspace{-2pt}\overline{P}_{(d_{z}, d_{w})})}$ is Hermitian. 
\end{enumerate} 
\end{remark} 

\section{Proof of Theorem \ref{thm:conj}} 
\label{sec:thmpf}

We begin this section with the statement and the proof of a lemma, that will be useful in the proof of Theorem \ref{thm:conj}. 

\begin{lemma} 
\label{lem:rho} 
Let $P_{(d_{w}, d_{z})}$ be a restrictive polynomial correspondence of bidegree $(d_{w}, d_{z})$. Then, ${\rm Rank} \left(M_{P_{(d_{w}, d_{z})}}\right) = \rho$ \textit{iff} all the three following statements hold. 
\begin{enumerate} 
\item $P_{(d_{w}, d_{z})} (z, w)$ can be rewritten as $g_{1}(w) h_{1}(z) + \cdots + g_{\rho}(w) h_{\rho}(z)$, where 
\[ g_{r}(w) = \sum\limits_{n\, =\, 0}^{d_{z}} \alpha_{(n, r)} w^{n}\ \text{and}\ h_{r}(z) = \sum\limits_{m\, =\, 0}^{d_{w}} \beta_{(m, r)} z^{m}, \] 
with $\alpha_{(d_{z}, r)} \ne 0$ for some $1 \le r \le \rho$ and $\beta_{(d_{w}, r)} \ne 0$ for some $1 \le r \le \rho$. 
\item The collection of vectors $\left\{ \left( \beta_{(m, r)} \right)_{0\, \le\, m\, \le\, d_{w}} \right\}_{1\, \le\, r\, \le\, \rho}$ is linearly independent. 
\item The collection of vectors $\left\{ \left( \alpha_{(n, r)} \right)_{0\, \le\, n\, \le\, d_{z}} \right\}_{1\, \le\, r\, \le\, \rho}$ is linearly independent. 
\end{enumerate} 
\end{lemma} 

\begin{proof} 
At the outset, we urge the readers to observe that the case $\rho = 2$ is the same as the equivalence between the statements $2$ and $3$ in Theorem \ref{thm:tfae}. 

We will initially assume that the matrix $M_{P_{(d_{w}, d_{z})}}$ is of rank $\rho$ and rewrite the given correspondence $P_{(d_{w}, d_{z})} \left( z, w \right)$ as $\sum\limits_{n\, =\, 0}^{d_{z}} q_{n} (z) w^{n}$, where $q_{d_{z}} \nequiv 0$, each polynomial $q_{n}$ is of degree atmost $d_{w}$ and there exists at least one polynomial $q_{n}$ whose degree is equal to $d_{w}$. Moreover, the coefficients of the polynomial $q_{n}$ forms the $n$-th column of the matrix $M_{P_{(d_{w}, d_{z})}}$, for $0 \le n \le d_{z}$. Then, the matrix has $\rho$ many column vectors, say $\left\{ {\rm Col}_{n_{r}} \right\}_{1\, \le\, r\, \le\, \rho}$ that form a basis for the column space, where $0 \le n_{r} \le d_{z}$. Thus, we have $q_{n} (z) = \sum\limits_{r\, =\, 1}^{\rho} \alpha_{(n, r)} q_{n_{r}} (z)$, thereby the correspondence can be written as 
\[ P_{(d_{w}, d_{z})} \left( z, w \right)\ =\ \sum_{n\, =\, 0}^{d_{z}} \sum_{r\, =\, 1}^{\rho} \alpha_{(n, r)} q_{n_{r}} (z) w^{n}\ =\ \sum_{r\, =\, 1}^{\rho} q_{n_{r}} (z) \sum_{n\, =\, 0}^{d_{z}} \alpha_{(n, r)} w^{n}. \]
We now define $g_{r} (w) = \sum\limits_{n\, =\, 0}^{d_{z}} \alpha_{(n, r)} w^{n}$ and $h_{r} (z) = q_{n_{r}} (z) = \sum\limits_{m\, =\, 0}^{d_{w}} \beta_{(m, r)} z^{m}$, say, so that the first statement is obtained. By our choice of the column vectors, $\left\{ {\rm Col}_{n_{r}} \right\}_{1\, \le\, r\, \le\, \rho}$, we have the collection $\left\{ \left( \beta_{(m, r)} \right)_{0\, \le\, m\, \le\, d_{w}} \right\}_{1\, \le\, r\, \le\, \rho}$ to be linearly independent. Further, for the matrix, say $N$ of size $(d_{z} + 1) \times \rho$ consisting of entries $\alpha_{(n, r)}$, observe that for $1 \le s \le \rho,\ \alpha_{(n_{s}, r)} = 1$ whenever $s = r$ and $0$ otherwise, since $q_{n_{s}} (z) = \sum\limits_{r\, =\, 1}^{\rho} \alpha_{(n_{r}, r)} q_{n_{r}} (z)$. Thus there exists a $\rho \times \rho$ submatrix of $N$ with a non-zero determinant, implying that the collection $\left\{ \left( \alpha_{(n, r)} \right)_{0\, \le\, n\, \le\, d_{z}} \right\}_{1\, \le\, r\, \le\, \rho}$ is linearly independent.

Conversely, suppose all three statements in the Lemma hold, then the matrix associated to the correspondence $P_{(d_{w}, d_{z})}$ is given by 
\begin{eqnarray*} 
M_{P_{(d_{w}, d_{z})}} & = & \begin{pmatrix} \displaystyle\sum_{r\, =\, 1}^{\rho} \alpha_{(d_{z}, r))} \beta_{(d_{w}, r))} & \displaystyle\sum_{r\, =\, 1}^{\rho} \alpha_{(d_{z} - 1, r))} \beta_{(d_{w}, r))} & \cdots & \displaystyle\sum_{r\, =\, 1}^{\rho} \alpha_{(0, r))} \beta_{(d_{w}, r))} \\ 
\vspace{+5pt}  
\displaystyle\sum_{r\, =\, 1}^{\rho} \alpha_{(d_{z}, r))} \beta_{(d_{w} - 1, r))} & \displaystyle\sum_{r\, =\, 1}^{\rho} \alpha_{(d_{z} - 1, r))} \beta_{(d_{w} - 1, r))} & \cdots & \displaystyle\sum_{r\, =\, 1}^{\rho} \alpha_{(0, r))} \beta_{(d_{w} - 1, r))} \\ 
\vspace{+5pt}  
\vdots & \vdots & \ddots & \vdots \\ 
\vspace{+5pt}  
\displaystyle\sum_{r\, =\, 1}^{\rho} \alpha_{(d_{z}, r))} \beta_{(0, r))} & \displaystyle\sum_{r\, =\, 1}^{\rho} \alpha_{(d_{z} - 1, r))} \beta_{(0, r))} & \cdots & \displaystyle\sum_{r\, =\, 1}^{\rho} \alpha_{(0, r))} \beta_{(0, r))} \end{pmatrix} 
\end{eqnarray*} 
\begin{eqnarray*} 
& = & \begin{pmatrix} \beta_{(d_{w}, 1)} & \beta_{(d_{w}, 2)} & \cdots & \beta_{(d_{w}, \rho)} \\ 
\vspace{+5pt}  
\beta_{(d_{w} - 1, 1)} & \beta_{(d_{w} - 1, 2)} & \cdots & \beta_{(d_{w} - 1, \rho)} \\ 
\vspace{+5pt}  
\vdots & \vdots & \ddots & \vdots \\ 
\vspace{+5pt}  
\beta_{(0, 1)} & \beta_{(0, 2)} & \cdots & \beta_{(0, \rho)} 
\end{pmatrix} \begin{pmatrix} \alpha_{(d_{z}, 1)} & \alpha_{(d_{z} - 1, 1)} & \cdots & \alpha_{(0, 1)} \\ 
\vspace{+5pt}  
\alpha_{(d_{z}, 2)} & \alpha_{(d_{z} - 1, 2)} & \cdots & \alpha_{(0, 2)} \\ 
\vspace{+5pt}  
\vdots & \vdots & \ddots & \vdots \\ 
\vspace{+5pt}  
\alpha_{(d_{z}, \rho)} & \alpha_{(d_{z} - 1, \rho)} & \cdots & \alpha_{(0, \rho)} 
\end{pmatrix}. 
\end{eqnarray*} 
Since each of the matrices in the above product is of rank $\rho$, the full rank factorisation theorem, as one may find in \cite{hj:2013, po:1999}, then implies that ${\rm Rank} \left(M_{P_{(d_{w}, d_{z})}}\right) = \rho$. 
\end{proof} 

We now prove Theorem \ref{thm:conj}. 

\begin{proof} 
In order to prove Theorem \ref{thm:conj}, we consider an irreducible restrictive polynomial correspondence $Q_{(d_{w}', d_{z}')}$ and prove that for any $d > 1$, the matrix pertaining to the reducible restrictive polynomial correspondence $P_{(d d_{w}', d d_{z}')} = \left[ Q_{(d_{w}', d_{z}')} \right]^{d}$ given by $M_{P_{(d d_{w}', d d_{z}')}}$ has rank $d + 1$. In fact, owing to Lemma \ref{lem:rho}, it is sufficient to prove that 
\[ P_{(d d_{w}', d d_{z}')} (z, w)\ =\ g_{0}(w) h_{0}(z) + \cdots + g_{d}(w) h_{d}(z), \]
where the coefficients of $g_{r}$ and $h_{r}$ satisfy the three conditions stated therein. 

Since $Q_{(d_{w}', d_{z}')}$ is an irreducible restrictive polynomial correspondence, using Theorem \ref{thm:tfae} and Lemma \ref{lem:rho}, we have that $Q_{(d_{w}', d_{z}')} (z, w) = p_{1}(w) q_{1}(z) + p_{2}(w) q_{2}(z)$, where 
\[ p_{j}(w) = \sum_{n\, =\, 0}^{d_{z}'} \alpha_{(n, j)} w^{n}\ \ \text{and}\ \ q_{j}(z) = \sum_{m\, =\, 0}^{d_{w}'} \beta_{(m, j)} z^{m}\ \ \text{for}\ j = 1, 2; \] 
such that 
\begin{enumerate} 
\item[(1)] either $p_{1}$ or $p_{2}$ has degree $d_{z}'$; 
\item[(2)] either $q_{1}$ or $q_{2}$ has degree $d_{w}'$; 
\item[(3)] there exists no $c_{1} \in \mathbb{C}$ such that $p_{1} = c_{1} p_{2}$ and 
\item[(4)] there exists no $c_{2} \in \mathbb{C}$ such that $q_{1} = c_{2} q_{2}$. 
\end{enumerate} 

Then, 
\begin{eqnarray*} 
P_{(d d_{w}', d d_{z}')} (z, w) & = & \left[ p_{1}(w) q_{1}(z) + p_{2}(w) q_{2}(z) \right]^{d} \\ 
& = & \sum_{r\, =\, 0}^{d} \binom{d}{r} \left( p_{1}(w) \right)^{d - r} \left( p_{2}(w) \right)^{r} \left( q_{1}(z) \right)^{d - r} \left( q_{2}(z) \right)^{r}. \end{eqnarray*} 
Define $g_{r}(w) = \displaystyle\binom{d}{r} \left( p_{1}(w) \right)^{d - r} \left( p_{2}(w) \right)^{r}$ and $h_{r}(z) = \left( q_{1}(z) \right)^{d - r} \left( q_{2}(z) \right)^{r}$. Owing to conditions (1) and (2) mentioned above, we have that at least one of the polynomials $g_{r}$'s has degree $dd_{z}'$ and that at least one of the polynomials $h_{r}$'s has degree $dd_{w}'$. 

Thus, what remains to be proven is that the collection of $(d + 1)$ vectors, each being a $(dd_{z}' + 1)$-tuple determined by the coefficients of the polynomial $g_{r}$ is linearly independent for $0 \le r \le d$ and so is the corresponding collection of $(d + 1)$ vectors, each being a $(dd_{w}' + 1)$-tuple determined by $h_{r}$. To do the same, we consider the matrix, say $\mathcal{G}$ of size $(dd_{z}' + 1) \times (d + 1)$, where each column is filled with the coefficients of $g_{r}$ in order, for $0 \le r \le d$ and prove that there exists a submatrix of $\mathcal{G}$ of order $(d + 1)$ whose determinant is non-zero. Observe that the entries in row $t$ of the matrix $\mathcal{G}$ correspond to the coefficients of $w^{dd_{z}' - t}$ in each of the polynomials $g_{r}$'s. Thus, the entry corresponding to row $t$ and column $r$ of the matrix $\mathcal{G}$ is given by 
\[ \mathcal{G}_{(t, r)}\ \ =\ \ \binom{d}{r} \sum_{k_{1} + k_{2}\, =\, dd_{z}' - t} \left[ \text{Coefficients of}\ w^{k_{1}}\ \text{in}\ p_{1}^{d - r} \right] \left[ \text{Coefficients of}\ w^{k_{2}}\ \text{in}\ p_{2}^{r} \right]. \]  

Owing to condition $(3)$, as mentioned above, there exists $0 \le n_{1}, n_{2} \le d_{z}'$ such that $\alpha_{(n_{1}, 1)} \alpha_{(n_{2}, 2)} - \alpha_{(n_{2}, 1)} \alpha_{(n_{1}, 2)} \ne 0$. Without loss of generality, we shall assume that $n_{1} = d_{z}'$ and $n_{2} = d_{z}' - 1$. Then, considering the square submatrix of size $(d + 1)$ formed by the first $(d + 1)$ rows of $\mathcal{G}$, we note that its determinant is equal to 
\[ \prod_{r\, =\, 0}^{d} \binom{d}{r} \left( \alpha_{(d_{z}', 1)} \alpha_{(d_{z}' - 1, 2)} - \alpha_{(d_{z}' - 1, 1)} \alpha_{(d_{z}', 2)} \right)^{\frac{d (d + 1)}{2}}\ \ \ne\ \ 0. \] 
One way to compute this determinant is to factorise the square submatrix as $L U P$ where $L$ is a lower triangular matrix with diagonal entries $1$, $P$ is a permutation matrix and $U$ is an upper triangular matrix with diagonal entries $\displaystyle{\binom{d}{r} \alpha_{(d_{z}', 1)}^{d - 2r} \left( \alpha_{(d_{z}', 1)} \alpha_{(d_{z}' - 1, 2)} - \alpha_{(d_{z}' - 1, 1)} \alpha_{(d_{z}', 2)} \right)^{r}}$ for $0 \le r \le d$. For any general $0 \le n_{1}, n_{2} \le d_{z}'$, one must choose an appropriate square submatrix of $\mathcal{G}$ of size $(d + 1)$ to conclude that its determinant is non-zero. This establishes that the coefficients of $g_{r}$ for $0 \le r \le d$ are linearly independent. 

The proof for the coefficients of $h_{r}$ for $0 \le r \le d$ being linearly independent follows \emph{mutatis mutandis}. 
\end{proof} 

\begin{example} 
Recall the polynomial correspondence $P_{(6, 4)}$, as stated in Example \eqref{ex:onesq}, that was obtained by squaring the correspondence $P_{(3, 2)}$, as given in Example \eqref{ex:one}. The matrix $M_{P_{(6, 4)}}$ pertaining to this correspondence has rank $3$. Further, defining 
\begin{eqnarray*} 
h_{1}(z) & = & - z^{6} - 2i z^{5} - 21 z^{4} - 24i z^{3} - 119 z^{2} - 22i z + 1, \\ 
h_{2}(z) & = & 10i z^{6} - (10+4i) z^{5} + (4+220i) z^{4} - (120+48i) z^{3} + (8+1210i) z^{2} \\ 
& & - (110+44i) z + 4, \\ 
h_{3}(z) & = & (25+2i) z^{6} - (22+12i) z^{5} + (566+44i) z^{4} - (264+144i) z^{3} + (3057+242i) z^{2} \\ 
& & - (242+132i) z + 16, \\ 
g_{1}(w) & = & w^{4} + \dfrac{-15792+11872i}{841} w + \dfrac{509432+776720 i}{24389}, \\ 
g_{2}(w) & = & w^{3} + \dfrac{5050+3528i}{841} w + \dfrac{149563-83748i}{24389}, \\ 
g_{3}(w) & = & w^{2} + \dfrac{34-56i}{29} w + \dfrac{-495-952i}{841}, 
\end{eqnarray*} 
one observes that $P_{(6,4)}$ can be rewritten as $g_{1}(w) h_{1}(z) + g_{2}(w) h_{2}(z) + g_{3}(w) h_{3}(z)$, as asserted by Lemma \ref{lem:rho}. 
\end{example} 

Finally, we conclude this manuscript with a more general example, that motivated the proof of Theorem \ref{thm:conj}. 

\begin{example} 
Consider the irreducible restrictive polynomial correspondence 
\[ P_{(d_{w}, d_{z})} (z,w)\ \ =\ \ z^{d_{w}} w^{d_{z}} + 1, \] 
where $d_{w}, d_{z} \ge 2$. Then, we see that for any $n \in \mathbb{Z}_{+}$, 
\[ P_{(d_{w}, d_{z})}^{n} (z,w) = z^{n d_{w}} w^{n d_{z}} + \binom{n}{1} z^{(n-1) d_{w}} w^{(n-1) d_{z}} + \binom{n}{2} z^{(n-2) d_{w}} w^{(n-2) d_{z}} + \cdots + 1. \] 
Then, setting $g_{r}(w) = w^{d_{z} (n - r)}$ and $h_{r}(z) = \displaystyle \binom{n}{r} z^{d_{w} (n - r)}$ for $0 \le r \le n$, we see that $P_{(d_{w}, d_{z})}^{n} (z,w)$ can be rewritten as $\displaystyle\sum_{r\, =\, 0}^{n} g_{r} (w) h_{r} (z)$. Also the matrix of size $(n d_{w} + 1) \times (n d_{z} + 1)$, namely $M_{P_{(d_{w}, d_{z})}^{n}}$ whose entries are given by,
     \[
    a_{j,k} = \begin{cases} 
        \displaystyle \binom{n}{n - r} & \text{for } j = d_{w} (n - r),\ k = d_{z} (n - r)\\
        0 & \text{otherwise } 
        \end{cases} 
  \] 
  with $0 \le r \le n$, has rank $(n + 1)$. 
\end{example}

\bigskip

\end{document}